\documentclass[12pt,a4paper]{article}
\usepackage{epsfig,psfrag,amsmath,amssymb,latexsym}
\usepackage{amscd }
\usepackage{amsfonts}
\usepackage{graphicx}
\usepackage{epstopdf}
\usepackage{bbm}
\usepackage{sectsty}
\usepackage{amsthm}

\sectionfont{\large}

\def\R{\mathbb{R}}
\def\N{\mathbb{N}}

\newtheorem{theorem}{Theorem}

\theoremstyle{remark}
\newtheorem{remark}{Remark}

\DeclareMathOperator{\Defi}{\mathrel{\mathop:}=}

\title{\large A TRANSFER PRINCIPLE FOR DEVIATIONS PRINCIPLES}
\author{Matthias L\"owe and Raphael Meiners\thanks{Research supported by German
National Academic Foundation}
\\
{\small \it Institute for Mathematical Statistics, University of M\"unster, Germany}}

\begin{document}
\maketitle
\vspace{-12pt}
\abstract{\noindent This note is not intended for publication. It provides a tool to infer moderate deviations principles for specific random variables from deviations principles for their Hubbard-Stratonovich transforms. This is needed for \cite{Loe2012}, wherefrom all notation is adopted.}

\vspace{2cm}
\noindent This article's sole purpose is to state and prove the following theorem:

\begin{theorem}\label{trans}
Let $m$ be a (local or global) minimum of $G$ and let $m$ be of type $k$ and strength $\lambda$.
\begin{itemize}
\item[(i)] Suppose that $$\left(P_{n,\beta}^h \circ \left(\frac{S_n-nm}{n^{\alpha}} + \frac{W}{n^{\alpha-\frac12}}\right)^{-1}\right)_{n \in \N}$$
satisfies for $\mu$-a.\,e.\ realization $h$ of $\mathbf{h}$ an MDP with speed $n^{1-2k(1-\alpha)}$ and rate function \begin{equation}\label{eq:rate function J}
J(x) ~\Defi~ J_{\lambda, k}(x) ~\Defi~ \frac{\lambda x^{2k}}{(2k)!}.
\end{equation}
Then, $$\left(P_{n,\beta}^\mathbf{h} \circ \left(\frac{S_n-nm}{n^{\alpha}}\right)^{-1}\right)_{n \in \N}$$ satisfies $\mu$-a.\,s.\ an MDP with speed $n^{1-2k(1-\alpha)}$ and rate function \begin{equation}\label{eq:rate function}
I(x) ~\Defi~ I_{k, \lambda, \beta}(x) ~\Defi~
\begin{cases}
\frac{x^2}{2 \sigma^2},& \text{ if } k =1,\\
\frac{\lambda x^{2k}}{(2 k)!},& \text{ if } k \geq2,\\
\end{cases}
\end{equation}
where $\sigma^2\Defi \lambda^{-1} - \beta^{-1} > 0$.
\item[(ii)] Suppose that 
\begin{equation*}
\beta ~>~ \frac{2 h}{b^2}.
\end{equation*}
Let $c$ be the supremum of all $x \in (0, (b  -\sqrt{2 h/\beta})/2]$ such that $m$ is the only minimum of $G$ in $[m-x,m+x]$ and fix $0<a<c$. Suppose that 
\begin{equation*}
\left(P_{n,\beta}^h \left(\frac{S_n-n m}{n^{\alpha}} + \frac{W}{n^{\alpha-\frac12}} \in \bullet \Big|\, \frac{S_n-n m}{n^{\alpha}} + \frac{W}{n^{\alpha-\frac12}} \in [-a n^{1-\alpha},a n^{1-\alpha}]\right)\right)_{n \in \N}
\end{equation*}
satisfies for $\mu$-a.\,e.\ realization $h$ of $\mathbf{h}$ an MDP with speed $n^{1-2k(1-\alpha)}$ and rate function $J$ given by \eqref{eq:rate function J}. Then, 
\begin{equation*}
\left(P_{n,\beta}^\mathbf{h} \left(\frac{S_n-n m}{n^{\alpha}} \in \bullet \Big|\, \frac{S_n}{n} \in [m-a,m+a]\right)\right)_{n \in \N}
\end{equation*}
satisfies $\mu$-a.\,s.\ an MDP with speed $n^{1-2k(1-\alpha)}$ and rate function $I$ given by \eqref{eq:rate function}.
\end{itemize}
\end{theorem}

\begin{remark}
Using Lebesgue's dominated convergence theorem we see
\begin{equation*}
G''(x) ~=~ \beta - \beta^2 \int_{\R}\frac{1}{\cosh^2(\beta(x+y))}d\nu(y) ~<~ \beta.
\end{equation*}
Since for $k=1$ there exists $x_{\text{max}} \in \R$ such that $\lambda = G''(x_{\text{max}})$, we immediately see $\beta > \lambda$ and, consequently, $\sigma^2 > 0$. 
\end{remark}

\begin{proof}[Proof of Theorem \ref{trans}]
Let $X_n \Defi (S_n-nm)/n^{\alpha}$ and $Y_n \Defi W/n^{\alpha-1/2}$. Choose $h$ such that $P_{n,\beta}^h \circ (X_n + Y_n)^{-1}$ in case (i) resp.\ $P_{n,\beta}^h (X_n + Y_n \in \bullet |\, X_n + Y_n \in [-a n^{1-\alpha},a n^{1-\alpha}])$ in case (ii) satisfy MDPs with speed $n^{2\alpha-1}$ and rate function $J$. This can be done with probability 1 due to the assumptions. Moreover, note that $P_{n,\beta}^h \circ Y_n^{-1}$ satisfies an MDP with speed $n^{2\alpha-1}$ and rate function $K(x) = \beta x^2 /2$ as it can be seen by means of the G\"{a}rtner-Ellis Theorem.

ad (i): Let us first consider the case $k>1$ and see that the influence of the Gaussian random variable vanishes for $n \to \infty$. Fix $\varepsilon > 0$ and note that
\begin{eqnarray*}
&& \limsup_{n \to \infty} \frac{1}{n^{1-2k(1-\alpha)}} \ln P_{n,\beta}^h\left(\left|X_n+Y_n-X_n\right|> \varepsilon\right)\\
&=& \limsup_{n \to \infty} \frac{1}{n^{1-2k(1-\alpha)}} \ln P_{n,\beta}^h\left(|Y_n|> \varepsilon\right)\\
&=& \limsup_{n \to \infty} \frac{n^{2(k-1)(1-\alpha)}}{n^{2\alpha-1}} \ln P_{n,\beta}^h\left(|Y_n|> \varepsilon\right)\\
&=& -\infty,
\end{eqnarray*}
since $2(k-1)(1-\alpha) > 0$, and, by the MDP for $P_{n,\beta}^h \circ Y_n^{-1}$,
\begin{equation*}
\lim_{n \to \infty} \frac{1}{n^{2\alpha-1}} \ln P_{n,\beta}^h\left(|Y_n|> \varepsilon\right) ~=~ -K(\varepsilon) ~<~ 0.
\end{equation*}
Therefore, $X_n + Y_n$ and $X_n$ are exponentially equivalent on the scale $n^{1-2k(1-\alpha)}$ and, thus, satisfy the same MDP (cf.\ Theorem 4.2.13 in \cite{Demb1998}). 

\noindent Now consider the case $k=1$. Note that it suffices to prove 
\begin{align}
\begin{split}
\lim_{n \to \infty}\frac{1}{n^{2\alpha-1}} \ln P_{n,\beta}^h (X_n \geq x) &= - I(x) \text{ for $x \geq 0$ and }\\
\lim_{n \to \infty}\frac{1}{n^{2\alpha-1}} \ln P_{n,\beta}^h (X_n \leq x) &= - I(x) \text{ for $x \leq 0$}
\end{split}
\label{eq:sufficient}
\end{align}
to gain the full MDP for $P_{n,\beta}^h \circ X_n^{-1}$, i.\,e.
\begin{eqnarray*}
\limsup_{n \to \infty} \frac{1}{n^{2\alpha-1}} \ln P_{n,\beta}^h (X_n \in C) &\le& - \inf_{x \in C} I(x) \text{ for every closed set $C \subset \R$,}\\
\liminf_{n\to \infty} \frac{1}{n^{2\alpha-1}} \ln P_{n,\beta}^h (X_n \in O) &\ge& - \inf_{x \in O} I(x) \text{ for every open set $O \subset \R$.}
\end{eqnarray*}
Indeed, if $0 \in C$, then $\inf_{x \in C} I(x) = 0$ and the upper bound holds trivially as $\ln P_{n,\beta}^h (X_n \in C)$ is always non-positive. On the other hand, if $0 \notin C$, we can define $a \Defi dist(C, \{0\})$, which is positive as $C$ is closed. Using \eqref{eq:sufficient} we obtain the general upper bound
\begin{eqnarray*}
&& \limsup_{n \to \infty} \frac{1}{n^{2\alpha-1}} \ln P_{n,\beta}^h (X_n \in C)\\
&\leq& \limsup_{n \to \infty} \frac{1}{n^{2\alpha-1}} \ln P_{n,\beta}^h (X_n \in (-\infty,-a] \cup [a,\infty))\\
&=& \limsup_{n \to \infty} \frac{1}{n^{2\alpha-1}} \ln (P_{n,\beta}^h (X_n \leq -a) + P_{n,\beta}^h (X_n \geq a))\\
&=& max\left\{\limsup_{n \to \infty} \frac{1}{n^{2\alpha-1}} \ln P_{n,\beta}^h (X_n \leq -a),\limsup_{n \to \infty} \frac{1}{n^{2\alpha-1}} \ln P_{n,\beta}^h (X_n \geq a)\right\}\\
&=& -I(a)\\
&=& - \inf_{x \in C} I(x)
\end{eqnarray*}
where we have made use of Lemma 1.2.15 from \cite{Demb1998} to derive the last but two line and of $I$'s monotonicity to derive the last line. To see that also the general lower bound follows from \eqref{eq:sufficient}, we first note that \eqref{eq:sufficient} implies the lower bound for arbitrary balls $B_\varepsilon(x) \Defi \{y \in \R|\, |x-y| < \varepsilon\}$ of radius $\varepsilon> 0$ centered at $x \in \R$:
\begin{itemize}
\item First case: $\varepsilon > |x|$\\
There exists $\delta > 0$ such that $B_\delta(0) \subset B_\varepsilon(x)$ and, consequently,
\begin{eqnarray*}
\liminf_{n\to \infty} \frac{1}{n^{2\alpha-1}} \ln P_{n,\beta}^h (X_n \in B_\varepsilon(x))
&\geq& \liminf_{n\to \infty} \frac{1}{n^{2\alpha-1}} \ln P_{n,\beta}^h (|X_n| < \delta)\\
&=& 0\\
&=& - \inf_{y \in B_\varepsilon(x)} I(y),
\end{eqnarray*}
where we have used in the second line that \eqref{eq:sufficient} implies $P_{n,\beta}^h (|X_n| < \delta) = 1-P_{n,\beta}^h (X_n \geq \delta)-P_{n,\beta}^h (X_n \leq -\delta) \rightarrow 1$.
\item Second case: $x \geq \varepsilon$\\
\eqref{eq:sufficient} yields for every $\delta > 0$ and $n$ sufficiently large
\begin{eqnarray*}
P_{n,\beta}^h (X_n \geq x - \varepsilon + \delta) &\geq& e^{n^{2\alpha-1}(-I(x-\varepsilon+\delta)-\delta)},\\
P_{n,\beta}^h (X_n \geq x + \varepsilon) &\leq& e^{n^{2\alpha-1}(-I(x+\varepsilon)+\delta)}.
\end{eqnarray*}
Since $I$ is a continuous function with $I(x+\varepsilon) > I(x-\varepsilon)$ we get
\begin{equation*}
-I(x+\varepsilon)+\delta+I(x-\varepsilon+\delta)+\delta ~<~ 0
\end{equation*}
for $\delta > 0$ sufficiently small and, therefore, 
\begin{eqnarray*}
&& \liminf_{n\to \infty} \frac{1}{n^{2\alpha-1}} \ln P_{n,\beta}^h (X_n \in B_\varepsilon(x))\\
&\geq& \liminf_{n\to \infty} \frac{1}{n^{2\alpha-1}} \ln (P_{n,\beta}^h (X_n \geq x - \varepsilon + \delta)-P_{n,\beta}^h (X_n \geq x + \varepsilon))\\
&\geq& \liminf_{n\to \infty} \frac{1}{n^{2\alpha-1}} \ln (e^{n^{2\alpha-1}(-I(x-\varepsilon+\delta)-\delta)}(1-e^{n^{2\alpha-1}(-I(x+\varepsilon)+\delta+I(x-\varepsilon+\delta)+\delta)}))\\
&=& -I(x-\varepsilon+\delta)-\delta
\end{eqnarray*}
for $\delta > 0$ sufficiently small. Taking $\delta \searrow 0$ yields 
\begin{equation*}
\liminf_{n\to \infty} \frac{1}{n^{2\alpha-1}} \ln P_{n,\beta}^h (X_n \in B_\varepsilon(x)) ~\geq~ -I(x-\varepsilon) ~=~ - \inf_{y \in B_\varepsilon(x)} I(y).
\end{equation*}
\item Third case: $x \leq - \varepsilon$\\
Again, \eqref{eq:sufficient} yields for every $\delta > 0$ and $n$ sufficiently large
\begin{eqnarray*}
P_{n,\beta}^h (X_n \leq x + \varepsilon - \delta) &\geq& e^{n^{2\alpha-1}(-I(x+\varepsilon-\delta)-\delta)},\\
P_{n,\beta}^h (X_n \leq x - \varepsilon) &\leq& e^{n^{2\alpha-1}(-I(x-\varepsilon)+\delta)},
\end{eqnarray*}
which implies 
\begin{eqnarray*}
&& \liminf_{n\to \infty} \frac{1}{n^{2\alpha-1}} \ln P_{n,\beta}^h (X_n \in B_\varepsilon(x))\\
&\geq& \liminf_{n\to \infty} \frac{1}{n^{2\alpha-1}} \ln (P_{n,\beta}^h (X_n \leq x + \varepsilon - \delta)-P_{n,\beta}^h (X_n \leq x - \varepsilon))\\
&\geq& \liminf_{n\to \infty} \frac{1}{n^{2\alpha-1}} \ln (e^{n^{2\alpha-1}(-I(x+\varepsilon-\delta)-\delta)}(1-e^{n^{2\alpha-1}(-I(x-\varepsilon)+\delta+I(x+\varepsilon-\delta)+\delta)})).
\end{eqnarray*} 
Since $I(x-\varepsilon) > I(x+\varepsilon)$ the continuity of $I$ yields
\begin{equation*}
-I(x-\varepsilon)+\delta+I(x+\varepsilon-\delta)+\delta ~<~ 0
\end{equation*}
for $\delta > 0$ sufficiently small and, consequently, 
\begin{equation*}
\liminf_{n\to \infty} \frac{1}{n^{2\alpha-1}} \ln P_{n,\beta}^h (X_n \in B_\varepsilon(x))
~\geq~ -I(x+\varepsilon-\delta)-\delta
\end{equation*}
for $\delta > 0$ sufficiently small. Once again, taking $\delta \searrow 0$ yields 
\begin{equation*}
\liminf_{n\to \infty} \frac{1}{n^{2\alpha-1}} \ln P_{n,\beta}^h (X_n \in B_\varepsilon(x)) ~\geq~ -I(x+\varepsilon) ~=~ - \inf_{y \in B_\varepsilon(x)} I(y).
\end{equation*}
\end{itemize}
The lower bound for open balls already gives the lower bound for arbitrary open sets. In fact, fix $G \subset \R$ open and let $x$ be an element of $G$ (the case $G = \emptyset$ holds trivially). Then, there exists $\varepsilon > 0$ s.\,t. $B_{\varepsilon}(x) \subset G$ and therefore
\begin{eqnarray*}
\liminf_{n\to \infty} \frac{1}{n^{2\alpha-1}} \ln P_{n,\beta}^h (X_n \in G) 
&\geq& \liminf_{n\to \infty} \frac{1}{n^{2\alpha-1}} \ln P_{n,\beta}^h (X_n \in B_\varepsilon(x))\\
&\geq& - \inf_{y \in B_\varepsilon(x)} I(y)\\
&\geq& -I(x) 
\end{eqnarray*}
for every $x \in G$. Taking the supremum over all $x \in G$ gives the desired lower bound. In a nutshell, we have seen that \eqref{eq:sufficient} yields the desired MDP and, therefore, we are left with a proof of \eqref{eq:sufficient}, which we start with a first observation:
\begin{eqnarray}\label{eq:obser}
I(x) 
&=& -\frac{\lambda \beta}{2(\beta-\lambda)}x^2\nonumber\\
&=& - \frac{\lambda\beta^2}{2(\beta-\lambda)^2}x^2+\frac{\beta\lambda^2}{2(\beta-\lambda)^2}x^2\nonumber\\
&=& - \frac{\lambda}{2}x_0^2+\frac{\beta}{2}(x_0-x)^2\nonumber\\
&=& - J(x_0)+K(x_0-x),
\end{eqnarray}
where $x_0 \Defi \frac{\beta}{\beta-\lambda} x$.\\

\noindent \emph{Upper bounds in \eqref{eq:sufficient}:}

\noindent Let us first consider the case $x \geq 0$. Since $X_n$ and $Y_n$ are independent, we have 
\begin{equation*}
P_{n,\beta}^h(X_n \geq x) P_{n,\beta}^h (Y_n \geq x_0-x) ~\leq~ P_{n,\beta}^h (X_n+Y_n \geq x_0) 
\end{equation*}
respectively
\begin{equation*}
P_{n,\beta}^h(X_n \geq x) ~\leq~ \frac{P_{n,\beta}^h (X_n+Y_n \geq x_0)}{P_{n,\beta}^h (Y_n \geq x_0-x)}.
\end{equation*}
Using the MDPs for $P_{n,\beta}^h \circ Y_n^{-1}$ and $P_{n,\beta}^h \circ (X_n +Y_n)^{-1}$, we see
\begin{eqnarray*}
\limsup_{n \to \infty}\frac{1}{n^{2\alpha-1}} \ln P_{n,\beta}^h (X_n \geq x) 
&\leq& \limsup_{n \to \infty}\frac{1}{n^{2\alpha-1}} \ln \frac{P_{n,\beta}^h (X_n+Y_n \geq x_0)}{P_{n,\beta}^h (Y_n \geq x_0-x)}\\
&=& -J(x_0)+K(x_0-x)\\
&=& I(x)
\end{eqnarray*}
where we have used the introductory observation \eqref{eq:obser}. In the remaining case $x\leq 0$, we have 
\begin{equation*}
P_{n,\beta}^h(X_n \leq x) ~\leq~ \frac{P_{n,\beta}^h (X_n+Y_n \leq x_0)}{P_{n,\beta}^h (Y_n \leq x_0-x)},
\end{equation*}
and with the same arguments as before we prove
\begin{equation*}
\limsup_{n \to \infty}\frac{1}{n^{2\alpha-1}} \ln P_{n,\beta}^h (X_n \leq x) ~\leq~ I(x).
\end{equation*}

\emph{Lower bounds in \eqref{eq:sufficient}:}

\noindent Again, we first consider the case $x \geq 0$. Using \eqref{eq:obser} and the continuity of $J$ it suffices to show
\begin{equation*}
\liminf_{n \to \infty}\frac{1}{n^{2\alpha-1}} \ln P_{n,\beta}^h (X_n \geq x) ~\geq~ -J(x_0 +\varepsilon)+K(x_0-x)
\end{equation*}
for every $\varepsilon > 0$ or, equivalently,
\begin{equation}\label{eq:tz}
\liminf_{n \to \infty}\frac{1}{n^{2\alpha-1}} \ln \frac{P_{n,\beta}^h (X_n \geq x) P_{n,\beta}^h (Y_n \geq x_0-x)}{P_{n,\beta}^h (X_n +Y_n \geq x_0+\varepsilon)} ~\geq~ 0
\end{equation}
for all $\varepsilon > 0$, where we have made use of the MDPs for $P_{n,\beta}^h \circ Y_n^{-1}$ and $P_{n,\beta}^h \circ (X_n +Y_n)^{-1}$. Fix $\varepsilon > 0$ and note that since
\begin{eqnarray*}
&& \frac{P_{n,\beta}^h (X_n \geq x) P_{n,\beta}^h (Y_n \geq x_0-x)}{P_{n,\beta}^h (X_n +Y_n \geq x_0+\varepsilon)}\\
&\geq& P_{n,\beta}^h (X_n \geq x, Y_n \geq x_0-x|\, X_n +Y_n \geq x_0 +\varepsilon)\\ 
&=& 1- P_{n,\beta}^h (X_n < x|\, X_n +Y_n \geq x_0 +\varepsilon) - P_{n,\beta}^h (Y_n < x_0-x|\, X_n +Y_n \geq x_0 +\varepsilon) 
\end{eqnarray*}
\eqref{eq:tz} follows once we have proved
\begin{eqnarray}
P_{n,\beta}^h (X_n < x|\, X_n +Y_n \geq x_0 +\varepsilon) &=& o(1),\label{eq:11}\\
P_{n,\beta}^h (Y_n < x_0-x|\, X_n +Y_n \geq x_0 +\varepsilon) &=& o(1) \label{eq:12}.
\end{eqnarray}

We start with a proof of \eqref{eq:11}. It is straightforward to see that
\begin{eqnarray*}
&& P_{n,\beta}^h (X_n < x_0(1-\sqrt{\lambda/\beta})|\, X_n +Y_n \geq x_0 +\varepsilon) \\
&=& \frac{P_{n,\beta}^h (X_n < x_0-x_0\sqrt{\lambda/\beta}, X_n +Y_n \geq x_0 +\varepsilon)}{P_{n,\beta}^h (X_n +Y_n \geq x_0 +\varepsilon)}\\
&\leq& \frac{P_{n,\beta}^h (Y_n \geq x_0\sqrt{\lambda/\beta} +\varepsilon)}{P_{n,\beta}^h (X_n +Y_n \geq x_0 +\varepsilon)},
\end{eqnarray*}
which again can be bounded using the MDPs for $P_{n,\beta}^h \circ Y_n^{-1}$ and $P_{n,\beta}^h \circ (X_n +Y_n)^{-1}$, which yield 
\begin{eqnarray*}
P_{n,\beta}^h (Y_n \geq x_0\sqrt{\lambda/\beta} +\varepsilon) &\leq& e^{-n^{2\alpha-1} (K(x_0\sqrt{\lambda/\beta} +\varepsilon) -\frac{\beta-\lambda}{4}\varepsilon^2)},\\
P_{n,\beta}^h (X_n +Y_n \geq x_0 +\varepsilon) &\geq& e^{-n^{2\alpha-1} (J(x_0 +\varepsilon) +\frac{\beta-\lambda}{4}\varepsilon^2)}
\end{eqnarray*}
for $n$ sufficiently large. Consequently, for $n$ sufficiently large
\begin{eqnarray*}
&& P_{n,\beta}^h (X_n < x_0(1-\sqrt{\lambda/\beta})|\, X_n +Y_n \geq x_0 +\varepsilon) \\
&\leq& e^{-n^{2\alpha-1} (K(x_0\sqrt{\lambda/\beta} +\varepsilon) -\frac{\beta-\lambda}{4}\varepsilon^2 -J(x_0 +\varepsilon) -\frac{\beta-\lambda}{4}\varepsilon^2)}\\
&=& e^{-n^{2\alpha-1} (\frac{\beta}{2}(x_0\sqrt{\lambda/\beta} +\varepsilon)^2 - \frac{\lambda}{2} (x_0 +\varepsilon)^2 -\frac{\beta-\lambda}{2}\varepsilon^2)}\\
&=& e^{-n^{2\alpha-1} \varepsilon x_0 \sqrt{\lambda}(\sqrt{\beta}-\sqrt{\lambda})}\\
&=& o(1).
\end{eqnarray*}
If $x = 0$, then $x_0(1-\sqrt{\lambda/\beta}) = x$ and we are done proving \eqref{eq:11}. Otherwise, $x_0(1-\sqrt{\lambda/\beta}) < x$, and we need to show that $P_{n,\beta}^h (x_0(1-\sqrt{\lambda/\beta})\leq X_n < x|\, X_n +Y_n \geq x_0 +\varepsilon)$ is a zero sequence. Let us devide, to that end, the interval $[x_0(1-\sqrt{\lambda/\beta}), x)$ into $M \in \N$ subintervals, each of lenght $\hat{x}/M$ where $\hat{x} \Defi x-x_0(1-\sqrt{\lambda/\beta}) > 0$. We get
\begin{eqnarray}
&& P_{n,\beta}^h (x_0(1-\sqrt{\lambda/\beta})\leq X_n < x|\, X_n +Y_n \geq x_0 +\varepsilon)\nonumber\\
&=& \sum_{i=1}^M \frac{P_{n,\beta}^h (x_0(1-\sqrt{\lambda/\beta})+\frac{i-1}{M}\hat{x} \leq X_n < x_0(1-\sqrt{\lambda/\beta})+\frac{i}{M}\hat{x}, X_n +Y_n \geq x_0 +\varepsilon)}{P_{n,\beta}^h (X_n +Y_n \geq x_0 +\varepsilon)}\nonumber\\
&\leq& \sum_{i=1}^M \frac{P_{n,\beta}^h (x_0(1-\sqrt{\lambda/\beta})+\frac{i-1}{M}\hat{x} \leq X_n) P_{n,\beta}^h(Y_n \geq x_0 +\varepsilon-x_0(1-\sqrt{\lambda/\beta})-\frac{i}{M}\hat{x})}{P_{n,\beta}^h (X_n +Y_n \geq x_0 +\varepsilon)}\nonumber\\
&=& \sum_{i=1}^M \frac{P_{n,\beta}^h (X_n \geq x_0(1-\sqrt{\lambda/\beta})+\frac{i-1}{M}\hat{x}) P_{n,\beta}^h(Y_n \geq x_0\sqrt{\lambda/\beta}+\varepsilon-\frac{i}{M}\hat{x})}{P_{n,\beta}^h (X_n +Y_n \geq x_0 +\varepsilon)}.\label{eq:13}
\end{eqnarray}
Using the upper bound for $\limsup_{n \to \infty}\frac{1}{n^{2\alpha-1}} \ln P_{n,\beta}^h (X_n \geq \cdot)$, which we have obtained before, and the MDPs for $P_{n,\beta}^h \circ Y_n^{-1}$ and $P_{n,\beta}^h \circ (X_n +Y_n)^{-1}$, we get for every $1\leq i \leq M$ and $n$ sufficiently large
\begin{eqnarray*}
P_{n,\beta}^h (X_n \geq x_0(1-\sqrt{\lambda/\beta})+ (i-1) \hat{x}/M) &\leq& e^{-n^{2\alpha-1}(I(x_0(1-\sqrt{\lambda/\beta})+\frac{i-1}{M}\hat{x})-\delta/3)}, \\
P_{n,\beta}^h(Y_n \geq x_0\sqrt{\lambda/\beta}+\varepsilon-i\hat{x}/M) &\leq& ^{-n^{2\alpha-1}(K(x_0\sqrt{\lambda/\beta}+\varepsilon-\frac{i}{M}\hat{x})-\delta/3)},\\
P_{n,\beta}^h (X_n +Y_n \geq x_0 +\varepsilon) &\geq& e^{-n^{2\alpha-1}(J(x_0 +\varepsilon)+\delta/3)},
\end{eqnarray*}
where $\delta \Defi (\beta-\lambda)\varepsilon^2/4 > 0$. Inserting these estimates in \eqref{eq:13} yields
\begin{eqnarray}
&& P_{n,\beta}^h (x_0(1-\sqrt{\lambda/\beta})\leq X_n < x|\, X_n +Y_n \geq x_0 +\varepsilon)\nonumber\\
&\leq& \sum_{i=1}^M e^{-n^{2\alpha-1}(I(x_0(1-\sqrt{\lambda/\beta})+\frac{i-1}{M}\hat{x})+K(x_0\sqrt{\lambda/\beta}+\varepsilon-\frac{i}{M}\hat{x})-J(x_0 +\varepsilon)-\delta)}\label{eq:14}
\end{eqnarray}
To find an upper bound for \eqref{eq:14}, we need to find the dominating summand. Note to this purpose that the function 
\begin{gather*}
f:\R\rightarrow\R,\\
z ~\mapsto~ I(x_0(1-\sqrt{\lambda/\beta})-\hat{x}/M+z \hat{x})+K(x_0\sqrt{\lambda/\beta}+\varepsilon-z\hat{x})
\end{gather*}
is decreasing on [0,1]:
\begin{eqnarray*}
&& \sup_{z \in [0,1]}f^{'}(z) \\
&=& \sup_{z \in [0,1]}\left\{\frac{\hat{x}}{\sigma^2}(x_0(1-\sqrt{\lambda/\beta})-\hat{x}/M+z \hat{x}) - \beta \hat{x}(x_0\sqrt{\lambda/\beta}+\varepsilon-z\hat{x})\right\}\\
&=& \frac{\hat{x}}{\sigma^2}(x_0(1-\sqrt{\lambda/\beta})-\hat{x}/M+\hat{x}) - \beta \hat{x}(x_0\sqrt{\lambda/\beta}+\varepsilon-\hat{x})\\
&<& \hat{x} \left(\frac{x_0(1-\sqrt{\lambda/\beta})}{\sigma^2}+\frac{\hat{x}}{\sigma^2} - \beta x_0 \sqrt{\lambda/\beta}+\beta \hat{x}\right)\\
&=& \hat{x} \left(\frac{x}{\sigma^2} + \beta x-\beta x_0\right)\\
&=& \hat{x} x \left(\frac{\beta \lambda}{\beta - \lambda} + \beta -\frac{\beta^2}{\beta-\lambda} \right)\\
&=& 0.
\end{eqnarray*}
This means that in \eqref{eq:14} the summand for $i = M$ is dominating and we get
\begin{eqnarray*}
&& P_{n,\beta}^h (x_0(1-\sqrt{\lambda/\beta})\leq X_n < x|\, X_n +Y_n \geq x_0 +\varepsilon) \\
&\leq& M e^{-n^{2\alpha-1}(I(x_0(1-\sqrt{\lambda/\beta})+\frac{M-1}{M}\hat{x})+K(x_0\sqrt{\lambda/\beta}+\varepsilon-\hat{x})-J(x_0 +\varepsilon)-\delta)}\nonumber\\
&=& M e^{-n^{2\alpha-1}\left(\frac{1}{2\sigma^2}(x-\hat{x}/M)^2+\frac{\beta}{2}(x_0-x+\varepsilon)^2-\frac{\lambda}{2}(x_0 +\varepsilon)^2-\delta\right)}\\
&=& M e^{-n^{2\alpha-1}\left(\delta -\frac{x\hat{x}}{\sigma^2 M} + \frac{\hat{x}^2}{2\sigma^2M^2} \right)},
\end{eqnarray*}
which is converging to 0 if $M$ is sufficiently large. That finishes the proof of \eqref{eq:11}.

To complete the proof of the upper bound in \eqref{eq:sufficient} for $x \geq 0$ we need to prove \eqref{eq:12}. As a start we note that 
\begin{eqnarray*}
&& P_{n,\beta}^h (Y_n < x_0(1-\sqrt{\lambda \sigma^2})|\, X_n +Y_n \geq x_0 +\varepsilon)\\
&\leq& \frac{P_{n,\beta}^h (X_n \geq x_0\sqrt{\lambda \sigma^2} +\varepsilon)}{P_{n,\beta}^h (X_n +Y_n \geq x_0 +\varepsilon)}\\
&\leq& e^{-n^{2\alpha-1}(I(x_0\sqrt{\lambda \sigma^2} +\varepsilon)-J(x_0 +\varepsilon)-\frac{\lambda^2}{2(\beta-\lambda)}\varepsilon^2)}\\
&=& e^{-n^{2\alpha-1}\varepsilon x_0 \lambda (\sqrt{\beta/(\beta-\lambda)}-1)}
\end{eqnarray*}
for $n$ sufficiently large, which converges to 0. Therein, we have used the MDP for $P_{n,\beta}^h \circ (X_n +Y_n)^{-1}$ and the upper bound for $\limsup_{n \to \infty}\frac{1}{n^{2\alpha-1}} \ln P_{n,\beta}^h (X_n \geq \cdot)$ obtained before. If $x=0$, then $x_0(1-\sqrt{\lambda \sigma^2}) = x_0-x$ and we have proved \eqref{eq:12}. Otherwise, $x_0(1-\sqrt{\lambda \sigma^2}) < x_0-x$ and we need to show that
\begin{equation*}
P_{n,\beta}^h (x_0(1-\sqrt{\lambda \sigma^2}) \leq Y_n < x_0-x|\, X_n +Y_n \geq x_0 +\varepsilon) ~=~ o(1).
\end{equation*}
We divide the interval $[x_0(1-\sqrt{\lambda \sigma^2}), x_0-x)$ into $M \in \N$ subintervals, each of lenght $\tilde{x}/M$ where $\tilde{x} \Defi x_0\sqrt{\lambda \sigma^2}-x > 0$. With the same ideas as before we show
\begin{eqnarray*}
&& P_{n,\beta}^h (x_0(1-\sqrt{\lambda \sigma^2}) \leq Y_n < x_0-x|\, X_n +Y_n \geq x_0 +\varepsilon)\\
&\leq& \sum_{i=1}^M \frac{P_{n,\beta}^h (X_n \geq x_0 \sqrt{\lambda \sigma^2} +\varepsilon -\frac{i}{M}\tilde{x}) P_{n,\beta}^h(Y_n \geq x_0 (1-\sqrt{\lambda \sigma^2})+ \frac{i-1}{M}\tilde{x})}{P_{n,\beta}^h (X_n +Y_n \geq x_0 +\varepsilon)}\\
&\leq& \sum_{i=1}^M e^{-n^{2\alpha-1}\left(I(x_0 \sqrt{\lambda \sigma^2} +\varepsilon -\frac{i}{M}\tilde{x})+K(x_0 (1-\sqrt{\lambda \sigma^2})+ \frac{i-1}{M}\tilde{x})-J(x_0 +\varepsilon)-\frac{\lambda^2}{4(\beta-\lambda)}\varepsilon^2\right)}\\
&\leq& M e^{-n^{2\alpha-1}\left(I(x_0 \sqrt{\lambda \sigma^2} +\varepsilon -\tilde{x})+K(x_0 (1-\sqrt{\lambda \sigma^2})+ \frac{M-1}{M}\tilde{x})-J(x_0 +\varepsilon)-\frac{\lambda^2}{4(\beta-\lambda)}\varepsilon^2\right)}\\
&=& M e^{-n^{2\alpha-1}\left(\frac{\lambda^2}{4(\beta-\lambda)}\varepsilon^2 - \frac{\beta(x_0-x)\tilde{x}}{M} + \frac{\beta \tilde{x}^2}{2 M^2}\right)}
\end{eqnarray*}
for $n$ sufficiently large, which again converges to 0 for $M$ sufficiently large. This ends the proof of \eqref{eq:12} yields the lower bound in \eqref{eq:sufficient} for $x \geq 0$.

We are left to prove the lower bound in equation \eqref{eq:sufficient} for $x \leq 0$. With the same arguments like the ones used in the case $x \geq 0$ it suffices to show
\begin{equation*}
\liminf_{n \to \infty}\frac{1}{n^{2\alpha-1}} \ln \frac{P_{n,\beta}^h (X_n \leq x) P_{n,\beta}^h (Y_n \leq x_0-x)}{P_{n,\beta}^h (X_n +Y_n \leq x_0-\varepsilon)} ~\geq~ 0
\end{equation*}
for all $\varepsilon > 0$. This follows from 
\begin{eqnarray*}
P_{n,\beta}^h (X_n > x|\, X_n +Y_n \leq x_0 -\varepsilon) &=& o(1),\\
P_{n,\beta}^h (Y_n > x_0-x|\, X_n +Y_n \leq x_0 -\varepsilon) &=& o(1).
\end{eqnarray*}
Again, with the same arguments as before we show 
\begin{eqnarray*}
P_{n,\beta}^h (X_n > x_0(1-\sqrt{\lambda/\beta})|\, X_n +Y_n \leq x_0 -\varepsilon) &\leq& e^{-n^{2\alpha-1} \varepsilon x_0(\sqrt{\lambda \beta}-\lambda)} \text{ and}\\
P_{n,\beta}^h (Y_n > x_0(1-\sqrt{\lambda\sigma^2})|\, X_n +Y_n \leq x_0 -\varepsilon) &\leq& e^{-n^{2\alpha-1}\varepsilon x_0 \lambda (\sqrt{\beta/(\beta-\lambda)}-1)}
\end{eqnarray*}
for $n$ sufficiently large. Thus, it is left to show
\begin{eqnarray*}
P_{n,\beta}^h (x < X_n \leq x_0(1-\sqrt{\lambda/\beta})|\, X_n +Y_n \leq x_0 -\varepsilon) &=& o(1) \text{ and}\\
P_{n,\beta}^h (x_0-x < Y_n \leq x_0(1-\sqrt{\lambda\sigma^2})|\, X_n +Y_n \leq x_0 -\varepsilon) &=& o(1).
\end{eqnarray*}
To that end, we divide the intervals $(x,x_0(1-\sqrt{\lambda/\beta})]$ resp.\ $(x_0-x,x_0(1-\sqrt{\lambda\sigma^2})]$ into $M \in \N$ subintervals, each of lenght $\hat{x}/M$ resp.\ $\tilde{x}/M$ where $\hat{x} = x_0(1-\sqrt{\lambda/\beta})-x$ and $\tilde{x} = x-x_0\sqrt{\lambda\sigma^2}$. Following the lines of the previous case, we get for $n$ sufficiently large
\begin{eqnarray*}
&& P_{n,\beta}^h (x < X_n \leq x_0(1-\sqrt{\lambda/\beta})|\, X_n +Y_n \leq x_0 -\varepsilon) \\
&\leq& M e^{-n^{2\alpha-1}\left(\frac{(\beta-\lambda)\varepsilon^2}{4} +\frac{x\hat{x}}{\sigma^2 M} + \frac{\hat{x}^2}{2\sigma^2M^2} \right)}
\end{eqnarray*}
and
\begin{eqnarray*}
&& P_{n,\beta}^h (x_0-x < Y_n \leq x_0(1-\sqrt{\lambda\sigma^2})|\, X_n +Y_n \leq x_0 -\varepsilon) \\
&\leq& M e^{-n^{2\alpha-1}\left(\frac{\lambda^2}{4(\beta-\lambda)}\varepsilon^2 + \frac{\beta(x_0-x)\tilde{x}}{M} + \frac{\beta \tilde{x}^2}{2 M^2}\right)},
\end{eqnarray*}
where we now have used that the dominant terms are the ones for $i=1$, i.\,e.
\begin{eqnarray*}
P_{n,\beta}^h (x < X_n \leq x+\hat{x}/M|\, X_n +Y_n \leq x_0 -\varepsilon) \text{ resp.}\\
P_{n,\beta}^h (x_0-x < Y_n \leq x_0-x+\tilde{x}/M|\, X_n +Y_n \leq x_0 -\varepsilon).
\end{eqnarray*}
That finishes the proof of the lower bound in equation \eqref{eq:sufficient} and thus all in all part (i).\\

ad (ii): Let $B_n \Defi [-a n^{1-\alpha},a n^{1-\alpha}]$ and choose the realization $h$ of $\mathbf{h}$ such that not only the MDP for $P_{n,\beta}^h (X_n + Y_n \in \bullet |\, X_n +Y_n \in B_n)$ holds, but also that
\begin{equation*}
P_{n,\beta}^h \left(\frac{S_n}{n} \in \bullet \right)
\end{equation*}
satisfies an LDP with speed $n$ and rate function $I_{\beta}^{\nu}(x) = \sup_{y\in\R}\{G(y)-\frac{\beta}{2}(x-y)^2\}-\inf_{w \in \R} G(w)$. This can be done with probability 1 due to \cite{Loew2012}. Using the G\"{a}rtner-Ellis Theorem, we see that 
\begin{equation*}
P_{n,\beta}^h \left( \frac{W}{\sqrt{n}} \in \bullet \right)
\end{equation*}
satisfies an LDP with speed $n$ and rate function $K(x) = \beta x^2/2$. Thus, a use of the contraction principle yields (cf.\ Exercise 4.2.7 in \cite{Demb1998}) that 
\begin{equation*}
\left(P_{n,\beta}^h \left(\frac{S_n}{n} + \frac{W}{\sqrt{n}} \in \bullet \right)\right)_{n \in \N}
\end{equation*}
satisfies an LDP with speed $n$ and rate function $N$ given by
\begin{eqnarray*}
N(x) 
&=& \inf_{\substack{(y,z) \in \R^2:\\ y+z=x}}\left(K(y)+I_{\beta}^{\nu}(z)\right)\\
&=& \inf_{z \in \R}\left(\frac{\beta}{2}(x-z)^2+\sup_{w \in \R}\left(G(w)-\frac{\beta}{2}(w-z)^2\right)\right) -\inf_{w \in \R} G(w)\\
&=& \frac{\beta}{2}x^2 + \inf_{z \in \R}\left(-\beta x z +\sup_{w \in \R}\left(G(w)-\frac{\beta}{2}w^2 +\beta w z\right)\right) -\inf_{w \in \R} G(w)\\
&=& \frac{\beta}{2}x^2 - \sup_{z \in \R}\left(x z -\sup_{w \in \R}\left(w z - \int_{\R}\ln \cosh[\beta(w+h)]d\nu(h)\right)\right) -\inf_{w \in \R} G(w)\\
&=& G(x)-\inf_{w \in \R} G(w)
\end{eqnarray*}
where we have used the duality lemma for Legendre-Fenchel transforms (cf.\ Lemma 4.5.8 in \cite{Demb1998}) to derive the last line. Since $X_n$ and $Y_n$ are \emph{not} independent under the measure $P_{n,\beta}^h (\bullet|\, X_n +Y_n \in B_n)$ we cannot proceed as in part (i). Instead, we start by showing that the MDP for $P_{n,\beta}^h (X_n + Y_n \in \bullet |\, X_n +Y_n \in B_n)$ transfers to the same MDP for 
\begin{equation*}
P_{n,\beta}^h (X_n + Y_n \in \bullet \Big|\, X_n \in B_n).
\end{equation*}
In order to do this, we show that the two sequences are exponentially equivalent on the scale $n^{1-2k(1-\alpha)}$, i.\,e.
\begin{equation}\label{eq:exp}
\limsup_{n \to \infty} \frac{1}{n^{1-2k(1-\alpha)}}\log \rho_n ~=~ -\infty
\end{equation}
where
\begin{equation*}
\rho_n ~\Defi~ \sup_{B \in \mathcal{B}(\R)}\big\{P_{n,\beta}^h (X_n +Y_n \in B |\, X_n \in B_n)-P_{n,\beta}^h (X_n + Y_n \in B |\, X_n + Y_n\in B_n)\big\}.
\end{equation*}
Note that for every $B \in \mathcal{B}(\R)$
\begin{eqnarray*}
&& P_{n,\beta}^h (X_n + Y_n \in B |\, X_n \in B_n)-P_{n,\beta}^h (X_n + Y_n \in B |\, X_n +Y_n \in B_n) \\
&\leq& P_{n,\beta}^h (|Y_n| > n^{(1-\alpha)/2}) + P_{n,\beta}^h (X_n + Y_n \in B,|Y_n| \leq n^{(1-\alpha)/2} |\, X_n \in B_n)\\
&& -P_{n,\beta}^h (X_n + Y_n \in B |\, X_n +Y_n \in B_n)\\
&=& P_{n,\beta}^h (|Y_n| > n^{(1-\alpha)/2})+ \left(\frac{1}{P_{n,\beta}^h (X_n \in B_n)}-\frac{1}{P_{n,\beta}^h (X_n +Y_n \in B_n)}\right) \\
&& \times \ P_{n,\beta}^h (X_n + Y_n \in B, X_n \in B_n, |Y_n| \leq n^{(1-\alpha)/2})\\
&& + \frac{P_{n,\beta}^h (X_n + Y_n \in B, X_n \in B_n, |Y_n| \leq n^{(1-\alpha)/2}) - P_{n,\beta}^h (X_n + Y_n \in B \cap B_n)}{P_{n,\beta}^h (X_n +Y_n \in B_n)}
\end{eqnarray*}
and consequently $\rho_n$ is bounded by
\begin{eqnarray*}
&& P_{n,\beta}^h (|Y_n| > n^{(1-\alpha)/2}) + \frac{P_{n,\beta}^h (X_n +Y_n \in B_n)-P_{n,\beta}^h (X_n\in B_n)}{P_{n,\beta}^h (X_n +Y_n \in B_n)}\vee0\\
&& + \frac{P_{n,\beta}^h (X_n + Y_n \in[-an^{1-\alpha}-n^{(1-\alpha)/2},-an^{1-\alpha}]\cup[an^{1-\alpha},an^{1-\alpha}+n^{(1-\alpha)/2}])}{P_{n,\beta}^h (X_n +Y_n \in B_n)}.
\end{eqnarray*}
Using Lemma 1.2.15 from \cite{Demb1998} \eqref{eq:exp} follows from proving that each of the three summands converges to $-\infty$ on a logarithmic scale of order $n^{1-2k(1-\alpha)}$.

First,
\begin{equation*}
\limsup_{n \to \infty} \frac{1}{n^{1-2k(1-\alpha)}}\log P_{n,\beta}^h (|Y_n| > n^{(1-\alpha)/2}) ~=~ -\infty
\end{equation*}
follows immediately from the standard estimate 
\begin{equation*}
P(Z > x) ~\leq~ \frac{1}{x \sqrt{2\pi}}\ e^{-\frac{x^2}{2}}
\end{equation*}
for a standard Gaussian $Z, x > 0$.

Second, with $\delta = b - c > 0$ it is
\begin{eqnarray*}
&& \frac{P_{n,\beta}^h (X_n +Y_n \in B_n)-P_{n,\beta}^h (X_n\in B_n)}{P_{n,\beta}^h (X_n +Y_n \in B_n)}\vee0\\
&=& \frac{P_{n,\beta}^h (S_n/n +W/\sqrt{n} \in [m-a,m+a])-P_{n,\beta}^h (S_n/n\in [m-a,m+a])}{P_{n,\beta}^h (S_n/n +W/\sqrt{n} \in [m-a,m+a])}\vee0\\
&\leq& \frac{P_{n,\beta}^h (S_n/n \in [m-a-\delta,m-a]\cup[m+a,m+a+\delta])}{P_{n,\beta}^h (S_n/n +W/\sqrt{n} \in [m-a,m+a])}\\
&& + \frac{P_{n,\beta}^h (|W/\sqrt{n}|>\delta)}{P_{n,\beta}^h (S_n/n +W/\sqrt{n} \in [m-a,m+a])}.
\end{eqnarray*}
Using Lemma 1.2.15 in \cite{Demb1998} we can again consider the two terms separately. We find
\begin{eqnarray*}
&&\lim_{n \to \infty} \frac{1}{n}\log \frac{P_{n,\beta}^h (S_n/n \in [m-a-\delta,m-a]\cup[m+a,m+a+\delta])}{P_{n,\beta}^h (S_n/n +W/\sqrt{n} \in [m-a,m+a])}\\
&=& -\inf_{x \in [m-a-\delta,m-a]\cup[m+a,m+a+\delta]}I_{\beta}^{\nu}(x)+\inf_{x \in [m-a,m+a]}N(x)\\
&\leq& -\inf_{x \in [m-a-\delta,m-a]\cup[m+a,m+a+\delta]}G(x)+G(m)\\
&<& 0
\end{eqnarray*}
and
\begin{eqnarray*}
&&\lim_{n \to \infty} \frac{1}{n}\log \frac{P_{n,\beta}^h (|W/\sqrt{n}|>\delta)}{P_{n,\beta}^h (S_n/n +W/\sqrt{n} \in [m-a,m+a])}\\
&=& -\frac{\beta \delta^2}{2} +\inf_{x \in [m-a,m+a]}G(x) - \inf_{x\in\R}G(x)\\
&=& -\frac{\beta \delta^2}{2} + h\\
&<& 0.
\end{eqnarray*}
Consequently,
\begin{equation*}
\limsup_{n \to \infty} \frac{1}{n^{1-2k(1-\alpha)}}\log \frac{P_{n,\beta}^h (X_n +Y_n \in B_n)-P_{n,\beta}^h (X_n\in B_n)}{P_{n,\beta}^h (X_n +Y_n \in B_n)}\vee0 ~=~ -\infty.
\end{equation*}

Finally, since $m$ is the only minimum of $G$ in $[m-a,m+a]$ we can choose $\tilde{a} > a$ such that $m$ is also the only minimum of $G$ in $[m-\tilde{a},m+\tilde{a}]$. Note that
\begin{eqnarray*}
&& \limsup_{n \to \infty} \frac{1}{n}\log P_{n,\beta}^h (X_n + Y_n \in[-an^{1-\alpha}-n^{(1-\alpha)/2},-an^{1-\alpha}]\\
&& \cup[an^{1-\alpha},an^{1-\alpha}+n^{(1-\alpha)/2}])\\
&\leq& \limsup_{n \to \infty} \frac{1}{n}\log P_{n,\beta}^h (S_n/n +W/\sqrt{n} \in [m-\tilde{a},m-a]\cup[m+a,m+\tilde{a}])\\
&=& -\inf_{x \in [m-\tilde{a},m-a]\cup[m+a,m+\tilde{a}]}G(x)+\inf_{x\in\R}G(x)\\
&<& -G(m)+\inf_{x\in\R}G(x)\\
&=& -\inf_{x \in [m-a,m+a]}G(x) + \inf_{x\in\R}G(x)\\
&=& \lim_{n \to \infty} \frac{1}{n} \log P_{n,\beta}^h (S_n/n + W/\sqrt{n} \in [m-a,m+a])\\
&=& \lim_{n \to \infty} \frac{1}{n} \log P_{n,\beta}^h (X_n + Y_n \in B_n)
\end{eqnarray*}
and therefore \eqref{eq:exp} follows.

Now that we know that $P_{n,\beta}^h (X_n + Y_n \in \bullet|\, X_n \in B_n)$ satisfies an MDP with speed $n^{1-2k(1-\alpha)}$ and rate function $J$ it is straightforward to prove the MDP for $P_{n,\beta}^h (X_n \in \bullet|\, X_n \in B_n)$. Since $X_n$ and $Y_n$ are independent under the measure $P_{n,\beta}^h ( \bullet |\, X_n \in B_n)$, the proof can be finished completely analogously to the proof of part (i) with $P_{n,\beta}^h$ replaced by $P_{n,\beta}^h ( \bullet |\, X_n \in B_n)$.
\end{proof}

\bibliographystyle{plain}
{\footnotesize \bibliography{transfer}}
\end{document}